\documentclass{m2an}
\usepackage{amssymb, amsthm}
\usepackage{mathtools}
\usepackage{bm}
\usepackage{booktabs}
\usepackage{cleveref}
\crefname{thrm}{theorem}{theorems}
\Crefname{thrm}{Theorem}{Theorems}

\newcommand{\PV}{\Pi_{V_h}}
\newcommand{\R}{\mathbb{R}}

\renewcommand{\d}{\mathrm{d}}

\DeclareMathOperator{\diag}{diag}

\DeclareMathOperator{\sym}{sym}
\DeclareMathOperator{\sech}{sech}
\newcommand{\norm}[1]{\left\lVert#1\right\rVert}

\begin{document}
%%-----------------------------
%%      the top matter
%%-----------------------------
\title{Energy stability of explicit--implicit--null time discretizations for a class of nonlinear degenerate parabolic equations with state-dependent coefficients}%\thanks{...}% At most 5 thanks
\author{Yue Wu} \address{Division of Applied Mathematics, Brown University, Providence, RI 02912, USA. \email{yue\_wu3@brown.edu}}
\author{Chi-Wang Shu} \address{Division of Applied Mathematics, Brown University, Providence, RI 02912, USA. \email{chi-wang\_shu@brown.edu}}
\date{November 5, 2026}
\begin{abstract}
We study a class of nonlinear degenerate parabolic equations with state-dependent coefficients admitting a scalar energy potential. The explicit--implicit--null (EIN) method avoids severe time step restrictions of explicit time stepping and costly Newton--Raphson iterations required by a fully implicit treatment of the nonlinearity. Using orthogonal projection of the Kirchhoff transform, we prove nonlinear energy stability for first- to third-order EIN time stepping coupled with a wide range of spatial discretizations that admit self-adjoint positive semidefinite discrete operators. Under an upper bound for the diffusion coefficient on relevant stage segments, two IMEX coefficient conditions yield sufficient stabilization bounds without any time step restriction. The estimates dissipate the original convex energy and allow degenerate diffusion. We also extend the analysis to diffusion systems with a scalar energy potential in a fixed positive definite metric, and to consistent positive-weight quadrature discretizations. Applications include porous-medium flow, nonlinear heat transfer in plasma, and multicomponent diffusion.
\end{abstract}
%
% \begin{resume} ... \end{resume} % French abstract
%
\subjclass{65M12, 65L06, 65L04, 35K65}
\keywords{degenerate parabolic equations, nonlinear diffusion, explicit--implicit--null time stepping, IMEX Runge--Kutta, energy stability, porous medium equation, nonlinear heat transfer }
\maketitle
%%-----------------------------
%%      your text
%%-----------------------------

\section{Introduction}\label{sec:objective}

\subsection{The model problem and energy stability}

Consider the initial boundary value problem of the nonlinear degenerate parabolic partial differential equation (PDE) with homogeneous Dirichlet boundary condition 
\begin{subequations}\label{eq:ivbp}
	\begin{alignat}{2}
		u_t - \nabla \cdot (a(u) \nabla u) &= 0, &\quad& (\mathbf{x}, t) \in \Omega \times (0,T], \label{eq:pde} \\
		u(0, \mathbf{x}) &= u_0(\mathbf{x}), &\quad& \mathbf{x} \in \Omega, \\
		u(t, \mathbf{x}) &= 0, &\quad& (\mathbf{x}, t) \in \partial\Omega \times (0,T]. \label{eq:bc}
	\end{alignat}
\end{subequations}
where \(\Omega \subset \mathbb R^d\) is a bounded Lipschitz domain and \(u: (0, T] \times \Omega \mapsto \mathbb R\) is the space-time unknown. The state-dependent coefficient function \(a \in C(\R, [0, +\infty))\) may have lower bound zero. The non-decreasing Kirchhoff transformation \cite{Kirchhoff1894} \(\phi \in C^1(\R, \R)\) defined by 
\begin{equation}
	\phi(u) \coloneqq \int_0^u a(v) \, \d v 
\end{equation}
turns the original PDE~\eqref{eq:pde} into
\begin{equation}\label{eq:Kirchhoff-transformed-PDE}
	u_t - \Delta \phi(u) = 0.
\end{equation}
The scalar energy potential \(\psi \in C^2(\R, [0, +\infty))\) given by 
\begin{equation}
	\psi(v) \coloneqq \int_{0}^v \phi(s) \, \d s
\end{equation}
satisfies $\psi' = \phi$ and $\psi'' = a \geq 0$. Since $\psi(0) = \psi'(0) = 0$, convexity also gives $\psi(v) \geq 0$. Define the convex energy functional
\begin{equation}\label{eq:nonlinearenergy}
	E(v) \coloneqq \int_\Omega \psi(v(\mathbf{x})) \, \d \mathbf{x}
\end{equation}
for any admissible function $v: \Omega \mapsto \R$ for which the integral is finite. Testing \eqref{eq:Kirchhoff-transformed-PDE} with \(\phi(u)\) and using integration by parts, any smooth solution \(u\) of \eqref{eq:ivbp} will satisfy 
\begin{subequations}
	\begin{equation}\label{eq:single-step-stability}
		\frac{\d}{\d t} E(u)
		=-\norm{\nabla\phi(u)}_{L^2(\Omega)}^2
		\leq 0, 
	\end{equation}
	and a time integral gives
	\begin{equation}\label{eq:space-time-stability}
		E(u(T)) + \norm{\nabla \phi(u)}_{L^2((0, T], L^2(\Omega))}^2
		= E(u_0).
	\end{equation}
\end{subequations}
We refer to discrete counterparts of \eqref{eq:single-step-stability} as one-step energy stability, and to those of \eqref{eq:space-time-stability} as space-time energy stability. 

The Kirchhoff formulation has been used in high-order discretizations of nonlinear degenerate diffusion equations. Liu, Shu, and Zhang \cite{10.1137/100791002} constructed finite difference weighted essentially nonoscillatory (WENO) schemes by discretizing the second derivative of the diffusion potential. Zhang, Liu, and Shu \cite{10.1137/110839230} developed maximum-principle-satisfying finite volume WENO schemes based on double cell averages. Arbogast, Huang, and Zhao \cite{10.1016/j.jcp.2019.108921} used the Kirchhoff transformation to construct finite volume WENO schemes on non-uniform meshes, with both explicit and implicit Runge--Kutta (RK) time discretizations. More recently, Yu, Cheng, Liu, and Shu \cite{10.1090/mcom/4167} established a Lax--Wendroff type theorem for doubly conservative schemes and proposed a local discontinuous Galerkin (LDG) discretization with semi-discrete $L^2$ stability in the degenerate case. Yu, Cheng, and Shu \cite{YCS2026DoCFV} further developed high-order doubly conservative finite volume WENO schemes on non-uniform and unstructured triangular meshes.

Nonlinear diffusion occurs in porous-medium flow, unsaturated flow, and heat transfer with conductivity that depends on temperature \cite{10.1063/1.1745233,10.1063/1.1745010,10.1103/PhysRev.89.977}. Explicit time discretization of a diffusion operator requires a time step that scales quadratically with the minimal spatial mesh size. Fully implicit discretization removes this parabolic stability restriction, but generally requires complicated Newton--Raphson or Picard iterations at each time step. The cost of repeatedly solving nonlinear systems can be substantial for high-order spatial discretizations and coupled problems, and the convergence of discrete nonlinear iterations may also be an issue.

\subsection{The EIN time discretization}

The EIN method for nonlinear diffusion treats the stiffness by adding and subtracting a linear diffusion term. For a constant \(a_0 \geq 0\), rewrite \eqref{eq:Kirchhoff-transformed-PDE} as 
\begin{equation}\label{eq:pde-ein-splitting}
	u_t = a_0 \Delta u + \Delta \{\phi(u) - a_0 u\}.
\end{equation}
For time discretizations, the first term is treated implicitly and the second explicitly. Thus the nonlinear function is evaluated at known old states, while each implicit stage requires only a symmetric positive definite linear solve. For the constant-coefficient model, suitable stabilization removes the parabolic time step restriction in the energy estimate \cite{10.1007/s11425-018-9524-x}. The method needs neither Newton--Raphson nor Picard iterations. Generally speaking, EIN is an application of an implicit--explicit (IMEX) method to the splitting \eqref{eq:pde-ein-splitting}; see \cite{10.1137/0732037,10.1016/S0168-92749700056-1} for the IMEX framework.

The addition and subtraction of linear diffusion appeared in the alternating-direction Galerkin methods of Douglas and Dupont \cite{10.1016/B978-0-12-358502-8.50009-8}. Duchemin and Eggers \cite{10.1016/j.jcp.2014.01.013} formulated the EIN approach for stiff PDEs and explained how a simple implicit linear operator can remove the restrictive stability constraint of an explicit calculation. Wang, Zhang, Wang, and Shu \cite{10.1007/s11425-018-9524-x} combined EIN time discretization with LDG \cite{10.1137/S0036142997316712} for nonlinear diffusion. Their stability and error analysis treats a constant-coefficient model, while their numerical examples include nonlinear coefficients. Related developments include IMEX-LDG methods for convection--diffusion \cite{10.1007/s10915-019-01072-4}, EIN finite difference, LDG, and spectral collocation schemes \cite{10.1016/j.jcp.2022.111314,10.4208/eajam.2022-271.090123}, variable-coefficient EIN stabilization \cite{10.1007/s42967-023-00359-6}, and EIN direct DG methods \cite{10.4208/nmtma.OA-2024-0076}. An application to the Richards equation is given in \cite{10.1016/j.compgeo.2026.108315}. 

Nonlinear stability of the LDG spatial discretization was established by Cockburn and Shu \cite[Sections 2.3 and 3]{10.1137/S0036142997316712}. For scalar nonlinear convection--diffusion equations with periodic boundary conditions, suitable numerical fluxes yield a semi-discrete $L^2$ estimate controlling the auxiliary diffusive variable and nonnegative interface dissipation. Their analysis allows nonnegative, possibly degenerate diffusion. In the isotropic case, the first-order reformulation uses the primitive $g$ with $g'=\sqrt{a}$, rather than the Kirchhoff transform $\phi'=a$. The estimate dissipates the quadratic energy $\frac{1}{2}\norm{u_h}_{L^2}^2$ in continuous time. The present analysis concerns fully discrete EIN dissipation of the nonlinear energy $E_h$ defined through $\psi'=\phi$.

Wang and Leung \cite{10.1016/j.cnsns.2024.108094} combine EIN with a partially explicit multiscale splitting for spatially heterogeneous, high-contrast nonlinear diffusion, with coefficient $\kappa(x,u)=\kappa_x(x)\kappa_u(u)$. Their stability analysis gives a Gronwall bound for a coefficient-weighted gradient quantity together with time increments, and the bound contains an exponential factor in time. The proof uses strict positivity and inverse bounds for the nonlinear diffusion factor, and does not cover degenerate diffusion.

\subsection{Our contribution}

We prove fully discrete nonlinear energy stability for first- to third-order EIN methods of \eqref{eq:ivbp}, coupled with a wide range of spatial discretizations. The estimates dissipate the discrete version of the convex energy \(E(u)\). The analysis uses a self-adjoint positive semidefinite discrete elliptic operator and the $L^2$ projection onto the approximation space. This includes symmetric continuous Galerkin (CG) methods \cite{10.1016/S0168-2024_08_70184-X}, mixed finite element methods (mixed FEM) \cite{10.1007/BFb0064470, 10.1007/BF01389710}, Bassi--Rebay 1 (BR1) \cite{10.1006/jcph.1996.5572}, Bassi--Rebay 2 (BR2) \cite{BassiEtAl1997BR2}, the method of Brezzi et al. \cite{BrezziEtAl1999Diffusion}, LDG \cite{10.1137/S0036142997316712}, symmetric interior penalty DG (SIPG) \cite{10.1137/0719052}, ultra-weak DG (UWDG) \cite{10.1090/S0025-5718-07-02045-5,10.1090/mcom/3927}, symmetric direct DG (DDG) \cite{10.4208/jcm.1307-m4273}, hybridizable DG (HDG) \cite{10.1090/S0025-5718-08-02123-6}, and weak Galerkin (WG) methods \cite{10.1016/j.cam.2012.10.003}, as well as spectral Galerkin methods \cite{10.1137/0915089}, mixed virtual element methods (mixed VEM) \cite{10.1051/m2an/2013138, 10.1051/m2an/2015067}, and symmetric two-point flux finite volume methods \cite{10.1016/S1570-86590007005-8}. Unified frameworks for the analysis and hybridization of these finite element methods are provided in \cite{10.1137/S0036142901384162, 10.1137/070706616}. Legendre--Gauss--Lobatto (LGL) and Fourier collocation \cite{10.1007/978-3-540-71041-7}, as well as centered finite differences \cite[Sections 12.5 and 13.2]{10.1137/1.9780898717938} and self-adjoint positive semidefinite diagonal-norm summation-by-parts (SBP) realizations \cite{10.1016/j.jcp.2004.03.001}, are covered by the quadrature-inner-product version described in \cref{rem:quadrature}. 

Following the stabilized IMEX analysis of \cite{10.1090/mcom/3950,10.4208/csiam-am.SO-2025-0095}, we reduce the energy estimate to two coefficient conditions,
\begin{equation}
	\sym H_0 \succeq 0, \quad
	\sym Q \succeq q I, \quad q>0.
\end{equation}
These conditions give energy stability for $a_0 \geq \frac{a_{\mathrm{max}}}{2q}$, where $a_{\mathrm{max}}$ bounds $a(\cdot)$ on the required state segments. Under this coefficient bound, all energy estimates are unconditional with respect to the time step and require no positive lower bound on the nonnegative coefficient $a(\cdot)$. A positive lower bound for $\sym H_0$ also controls the dissipation but is not necessary for the general stability result. We apply this criterion to an EIN1 scheme, the EIN2 scheme of \cite{10.1007/s11425-018-9524-x}, the EIN2-LWW scheme, and two EIN3 schemes. The proof also extends to systems of equations admitting a scalar energy potential in a fixed positive definite metric $H$, with $H A(\mathbf{u}) = \nabla^2\Psi(\mathbf{u})$. Consistent positive-weight quadrature and homogeneous Neumann or periodic boundary conditions can be treated as described in \cref{rem:quadrature}.

\Cref{sec:space,sec:stagecriterion} introduce the spatial setting and the general energy estimate. \Cref{sec:energy} treats the first- to third-order schemes. Generalizations to the system case, quadrature-based numerical implementations and other boundary conditions are discussed in \cref{sec:generalizations}. Some practical applications are listed in \cref{sec:applications}, followed by concluding remarks in \cref{sec:summary}.

\section{Spatial discretization}\label{sec:space}

Let $\mathcal T_h$ be a suitable mesh partition of $\Omega$, with mesh size denoted by $h$, and let $V_h\subset L^2(\Omega)\cap L^\infty(\Omega)$ be a finite-dimensional approximation space. The homogeneous Dirichlet boundary condition is incorporated in the discrete operator and, when appropriate, the space. No specific choice of basis, element shape, or local polynomial degree is needed in the analysis below. In particular, compatible nonmatching meshes and different local approximation spaces may be used. All spatial integrals are exact until \cref{rem:quadrature}. Write
\begin{equation*}
	(v,w)=\int_\Omega vw\,\d \mathbf{x},\quad \|v\|^2=(v,v),
\end{equation*}
with the Euclidean dot product for vector arguments. Denote the orthogonal projection onto $V_h$ by $\PV$.

\subsection{Discrete Laplacian}

Let $L_h: V_h \mapsto V_h$ approximate the Laplacian, and set \(K_h \coloneqq -L_h\). We assume self-adjointness and positive semidefiniteness,
\begin{equation}\label{eq:psd}
	K_h^* = K_h, \quad
	(K_h v_h, v_h) \geq 0 \quad \forall v_h \in V_h. 
\end{equation}
For example, a symmetric discrete elliptic form $b_h: V_h \times V_h \mapsto \R$ defines this operator by 
\begin{equation*}
  (K_h v_h, w_h) = b_h(v_h, w_h), \quad \forall v_h, w_h \in V_h.
\end{equation*}
We may choose from symmetric realizations of CG, mixed FEM, BR1, BR2, the Brezzi et al. method, LDG, SIPG, UWDG, DDG, HDG, or WG, with suitable numerical fluxes and stabilization parameters; auxiliary flux and trace variables must be eliminated before defining $K_h$, i.e., the primal operator. The constructions and relations among these elliptic methods are described in \cite{10.1137/S0036142901384162, 10.1137/070706616}. The assumptions \eqref{eq:psd}, rather than a particular realization, are used in the proofs. The dissipation, which can be viewed as a discrete counterpart of the squared \(H_0^1\) seminorm, is
\begin{equation}\label{eq:highdissipation}
	\mathcal D_h(v_h) \coloneqq (K_h v_h, v_h) \geq 0, \quad v_h\in V_h.
\end{equation}

\subsection{Spatial discretization using projected Kirchhoff transform}

The semi-discrete scheme for \eqref{eq:pde-ein-splitting} is
\begin{equation}
	\partial_t u_h = a_0 L_h u_h + L_h \{ \PV\phi(u_h) - a_0 u_h \}.
\end{equation}
The nonlinear spatial term is $L_h\PV\phi(u_h)$, with the Galerkin methods listed above used to construct $L_h$. The same fixed operator is used for both parts of the EIN splitting. The first-order method, for example, is 
\begin{equation}\label{eq:EIN1}
	\frac{u_h^{n+1} - u_h^n}{\tau} = L_h W_h^{n+1}, \quad
	W_h^{n+1} = \PV\phi(u_h^n) + a_0(u_h^{n+1} - u_h^n),
\end{equation}
or equivalently
\begin{equation}\label{eq:ein}
	(I + \tau a_0 K_h) u_h^{n+1}
	= u_h^n - \tau K_h \{\PV\phi(u_h^n) - a_0 u_h^n\}.
\end{equation}
Every implicit stage requires inversion of $I + \gamma \tau a_0 K_h$ for some stage coefficient $\gamma > 0$, which is positive definite for every $\tau \geq 0$. The projection evaluates the nonlinear function at a known state and applies a fixed linear operator. Thus no nonlinear iteration is required.

\subsection{Discrete energy functional}

Assume that $\psi\in C^2(I)$ on an open state interval $I$, $\psi'=\phi$, and $\psi''=a\geq0$. For the energy functional \eqref{eq:nonlinearenergy}, write $E_h:=E|_{V_h}$ for its restriction to the discrete approximation space $V_h$.

Consider $E_h$ on the open subset of $V_h$ consisting of functions whose essential ranges are compactly contained in $I$. The restriction is Fr\'echet differentiable, and its derivative satisfies
\begin{equation}\label{eq:energygradient}
	\langle D E_h(v_h),w_h\rangle_{V_h',V_h}
	= (\phi(v_h), w_h)
	= (\PV\phi(v_h), w_h), \quad v_h, w_h\in V_h.
\end{equation}
Here $\langle\cdot,\cdot\rangle_{V_h',V_h}$ denotes the natural duality pairing between $V_h'$ and $V_h$. Thus the $L^2$ gradient of $E_h$ is $\nabla_{V_h}E_h(v_h)=\PV\phi(v_h)$. Indeed, finite dimensionality and $V_h\subset L^\infty(\Omega)$ give a uniform bound on $a(\cdot)$ in a neighborhood of any such $v_h$. Taylor's remainder is therefore $O(\|w_h\|^2)$, proving Fr\'echet differentiability. In particular, \eqref{eq:energygradient} also gives the G\^ateaux derivative. Taking $w_h=\partial_t u_h$ yields
\begin{equation}\label{eq:semienergy}
	\frac{\d}{\d t} E_h(u_h)
	= -\mathcal D_h(\PV\phi(u_h)) \leq 0.
\end{equation}
Identity \eqref{eq:semienergy} is the semi-discrete counterpart of \eqref{eq:single-step-stability}. Integration over $(0,T]$ gives the corresponding space-time identity \eqref{eq:space-time-stability}.

For $v_h, d_h \in V_h$, with their pointwise segments in the state interval, the integrated Taylor formula is
\begin{equation}\label{eq:taylorexact}
	E_h(v_h + d_h) - E_h(v_h)
	= (\PV\phi(v_h), d_h) + \int_\Omega d_h(x)^2\int_0^1 (1-t) a(v_h(x) + t d_h(x)) \, dt \, \d \mathbf{x}.
\end{equation}
Consequently, if $0\leq a(v_h(x)+td_h(x))\leq a_{\mathrm{max}}$ for almost every $x\in\Omega$ and all $0\leq t\leq 1$, then
\begin{equation}\label{eq:taylorbound}
	E_h(v_h+d_h)-E_h(v_h)\leq(\PV\phi(v_h),d_h)+\frac{a_{\mathrm{max}}}{2}\|d_h\|^2.
\end{equation}
For unbounded coefficients \(a(\cdot)\), the bound \(a_{\mathrm{max}}\) must hold on all required pointwise state segments throughout $\Omega$.

\section{EIN time discretization and general energy stability}\label{sec:stagecriterion}

\subsection{Stage equations and elimination}

For an EIN time step from $u_h^n$ to $u_h^{n+1}$, set $u_h^{(0)}=u_h^n$ and $u_h^{n+1}=u_h^{(s)}$ and define the projected Kirchhoff transform
\begin{equation}
	\phi_h^{(j)} \coloneqq \PV\phi(u_h^{(j)}),\quad 0\leq j<s.
\end{equation}
Let $A^I = (a^I_{ij})_{i,j=0}^s$ and $A^E = (a^E_{ij})_{i,j=0}^s$ denote the coefficient arrays in the Butcher tabulae for the implicit and explicit parts of the IMEX-RK method, respectively. Both arrays include a row for the old state, indexed by zero. Row zero of each array is zero, and the last rows are the output weights. We assume the IMEX-RK methods considered here to have the following IMEX-DIRK structure. 

\begin{crtrn}[IMEX-DIRK structure]
	We assume that the implicit array \((a^I_{ij})_{i,j=0}^s\) is lower-triangular with positive diagonals, i.e., \(a_{ii}^I > 0\), and that the explicit array \((a^E_{ij})_{i,j=0}^s\) is strictly lower-triangular (which always holds true for the definition of ``explicit") with non-negative sub-diagonals, i.e., \(a_{i,i-1}^E \neq 0\) for \(1\leq i\leq s\). 
\end{crtrn}

In component form, the stage equations are
\begin{equation}\label{eq:generalimex-components}
	u_h^{(i)} = u_h^n - \tau K_h \left[a_0\sum_{j=0}^i a^I_{ij}u_h^{(j)} 
	+ \sum_{j=0}^{i-1} a^E_{ij} \left(\phi_h^{(j)}-a_0u_h^{(j)}\right)\right],
	\quad 1\leq i\leq s.
\end{equation}
We also write $\phi_h^n=\phi_h^{(0)}$. At stage $i$, the Kirchhoff transform is evaluated only at known states. The implicit solve uses the fixed linear operator $I+\tau a_0 a^I_{ii} K_h$. The temporal orders stated below refer to the classical order of the IMEX formula applied to the semi-discrete problem on a fixed mesh, assuming sufficient smoothness of its vector field and solution. The energy estimates require only the coefficient regularity and bounds specified in the corresponding theorems.

Collect the new states, preceding states, and preceding Kirchhoff transforms as
\begin{equation*}
	\begin{aligned}
		[u_h^\ast] & = (u_h^{(1)},\ldots,u_h^{(s)})^T,\\
		[u_h] & = (u_h^{(0)},\ldots,u_h^{(s-1)})^T,\\
		[\phi_h] & = (\phi_h^{(0)},\ldots,\phi_h^{(s-1)})^T.
	\end{aligned}
\end{equation*}
Let $\bm e$ denote the vector of ones in the temporal stage space, and set
\begin{equation*}
	\bm a_0^I=(a^I_{10},\ldots,a^I_{s0})^T,\quad
	B=(a^I_{ij})_{i,j=1}^s,\quad
	C=(a^E_{i,j-1})_{i,j=1}^s.
\end{equation*}
The stage equations can be written compactly as
\begin{equation}\label{eq:generalimex}
	[u_h^\ast] = \bm e u_h^n - \tau K_h \left\{ a_0\bm a_0^I u_h^n + a_0 B [u_h^\ast] + C\left([\phi_h] - a_0[u_h]\right)\right\},
\end{equation}
where $K_h$ acts componentwise on vectors of functions in $V_h$. Thus the $i$th entries of $[u_h^\ast]$ and $[u_h]$ are $u_h^{(i)}$ and $u_h^{(i-1)}$, respectively.

Define the stage increments and their vector by
\begin{equation*}
	d_h^{(i)} \coloneqq u_h^{(i)}-u_h^{(i-1)},\quad
	[d_h] = (d_h^{(1)},\ldots,d_h^{(s)})^T.
\end{equation*}
Let $T_{ij}=1$ for $j\leq i$ and zero otherwise, and define
\begin{equation}\label{eq:stageHQ}
	H_0 \coloneqq C^{-1} T, \quad Q \coloneqq (C^{-1} B - I) T + I.
\end{equation}
The current and preceding stage vectors satisfy $[u_h^\ast] = \bm e u_h^n + T [d_h]$ and $[u_h] = \bm e u_h^n + (T - I) [d_h]$. Equal row sums give $\bm a_0^I + B\bm e = C \bm e$, which cancels the old-state contribution of the added and subtracted diffusion. Multiplication by $C^{-1}$ gives
\begin{equation}\label{eq:stagebalance}
	H_0[d_h]=-\tau K_h\left([\phi_h]+a_0Q[d_h]\right).
\end{equation}
Define the stage potentials by
\begin{equation}\label{eq:stagepotentials}
	[z_h] \coloneqq H_0^{-1}\left([\phi_h]+a_0Q[d_h]\right),\quad
	[z_h] = (z_h^{(1)},\ldots,z_h^{(s)})^T.
\end{equation}
The inverse $H_0^{-1} = T^{-1}C$ acts only in the temporal stage space. These potentials are auxiliary quantities determined by the computed stages. Since temporal arrays commute with the componentwise spatial operator, \eqref{eq:stagebalance} and \eqref{eq:stagepotentials} yield
\begin{equation}\label{eq:stageelimination}
	\begin{aligned}
		d_h^{(i)}&=-\tau K_hz_h^{(i)},\\
		\phi_h^{(i-1)}&=\sum_{j=1}^s(H_0)_{ij}z_h^{(j)}
		-a_0\sum_{j=1}^sQ_{ij}d_h^{(j)},
	\end{aligned}
	\quad 1\leq i\leq s.
\end{equation}

\subsection{Energy theorem and accumulated dissipation}

For a real matrix $X$, write $\sym X=\frac{X + X^T}{2}$. 

\begin{thrm}[Energy stability for general EIN from IMEX-RK]\label{thm:stagecriterion}
	Let the EIN method for \eqref{eq:pde-ein-splitting} be generated from an IMEX-RK method with Butcher tabulae \(A^I\) and \(A^E\) and having auxiliary matrices \(H_0\) and \(Q\) defined by \eqref{eq:stageHQ}. Suppose
	\begin{equation}\label{assumption:stabilization-bounds}
		0\leq a(u_h^{(i-1)}(x)+t\,d_h^{(i)}(x))\leq a_{\mathrm{max}}
		\quad \text{for a.e. }x\in\Omega, \quad 0\leq t\leq1, \quad 1\leq i\leq s,
	\end{equation}
	and suppose the coefficient matrices satisfy
	\begin{equation}\label{eq:HQcondition}
		\sym H_0\succeq \eta I,\quad
		\sym Q\succeq qI,\quad \eta\geq0,\quad q>0.
	\end{equation}
	If $a_0\geq \frac{a_{\mathrm{max}}}{2q}$, then
	\begin{equation}\label{eq:generalenergy}
		\boxed{E_h(u_h^{n+1})+\tau\eta\sum_{i=1}^s\mathcal D_h(z_h^{(i)})
		+(a_0q-\frac{a_{\mathrm{max}}}{2})\sum_{i=1}^s\|d_h^{(i)}\|^2\leq E_h(u_h^n).}
	\end{equation}
\end{thrm}

\begin{proof}
	Orthogonality of $\PV$ and \eqref{eq:taylorbound} give
	\begin{equation*}
		E_h(u_h^{(i-1)}+d_h^{(i)})-E_h(u_h^{(i-1)})\leq(\phi_h^{(i-1)},d_h^{(i)})+\frac{a_{\mathrm{max}}}{2}\|d_h^{(i)}\|^2.
	\end{equation*}
	After summation and substitution of \eqref{eq:stageelimination},
	\begin{equation}\label{eq:generalenergyexact}
		E_h(u_h^{n+1})-E_h(u_h^n)
		\leq -\tau\sum_{i,j}(H_0)_{ij}(z_h^{(i)},K_h z_h^{(j)})
		- \sum_{i,j}(a_0Q_{ij}-\frac{a_{\mathrm{max}}}{2}\delta_{ij})(d_h^{(i)},d_h^{(j)}).
	\end{equation}
	The matrices $G_{ij}=(z_h^{(i)},K_hz_h^{(j)})$ and $M_{ij}=(d_h^{(i)},d_h^{(j)})$ are symmetric positive semidefinite. Indeed, for any $\bm\xi\in\mathbb R^s$,
	\begin{equation*}
		\bm\xi^T G\bm\xi
		=\mathcal D_h\left(\sum_i\xi_i z_h^{(i)}\right)\geq0,
	\end{equation*}
	and the assertion for $M$ follows in the same way from the $L^2$ inner product. Thus only $\sym H_0$ and $\sym Q$ contribute, and
	\begin{equation*}
		\sum_{i,j}(H_0)_{ij}G_{ij}\geq\eta\sum_i\mathcal D_h(z_h^{(i)}),\quad
		\sum_{i,j}Q_{ij}M_{ij}\geq q\sum_i\|d_h^{(i)}\|^2.
	\end{equation*}
	These bounds prove \eqref{eq:generalenergy}. The proof requires only self-adjointness and positive semidefiniteness of $K_h$.
\end{proof}

Under the presumed stabilization bounds \eqref{assumption:stabilization-bounds}, all energy estimates below hold for arbitrary positive time steps and require no strictly positive lower bound on $a(\cdot) \geq 0$.

Estimate~\eqref{eq:generalenergy} is a one-step discrete counterpart of \eqref{eq:single-step-stability}: when $\eta>0$, it controls the dissipation of the stage potentials. The $L^2$ increment term is additional numerical dissipation. For variable steps $\tau_n > 0$, fixed $a_0$, and a common coefficient bound $a_{\mathrm{max}}$, let $d_h^{n,(i)}$ and $z_h^{n,(i)}$ denote the stage increments and potentials in time step $n$. Summation yields
\begin{equation}\label{eq:generalenergysum}
	E_h(u_h^N)+\sum_{n=0}^{N-1}\left\{
	\tau_n\eta\sum_i\mathcal D_h(z_h^{n,(i)})+
	(a_0 q - \frac{a_{\mathrm{max}}}{2})\sum_i\|d_h^{n,(i)}\|^2\right\}\leq E_h(u_h^0).
\end{equation}
This is the space-time discrete counterpart of \eqref{eq:space-time-stability}. If $\eta>0$, the stage-potential dissipation remains controlled at the threshold $a_0=\frac{a_{\mathrm{max}}}{2q}$.

This stage-matrix argument follows the stabilized IMEX energy analysis of \cite{10.1090/mcom/3950,10.4208/csiam-am.SO-2025-0095}. In the notation of \cite{10.4208/csiam-am.SO-2025-0095}, $D_E=H_0$ and $D_{EI}=Q-\frac{I}{2}$. Here the orthogonal projection gives the energy derivative directly, and Taylor's bound gives the coefficient $a_0q-\frac{a_{\mathrm{max}}}{2}$. The scalar $a_{\mathrm{max}}$ is an upper bound for the actual diffusion coefficient $a$ on the specified pointwise segments; no coefficient freezing is made in the scheme.

\section{Energy stability for first- to third-order EIN methods}\label{sec:energy}

\subsection{The EIN1 method}

The EIN1 scheme \eqref{eq:EIN1} is 
\begin{equation}\label{eq:highenergyupdate}
	\delta u_h=\tau L_hW_h^{n+1},\quad
	\delta u_h=u_h^{n+1}-u_h^n,\quad W_h^{n+1}=\phi_h^n+a_0\delta u_h.
\end{equation}

\begin{thrm}[Energy stability of EIN1]\label{thm:highenergy}
	Suppose $\phi, \psi$ are defined on the relevant state interval and
	\begin{equation}\label{eq:energycoefficient}
		0\leq a(u_h^n(x)+t\delta u_h(x))\leq a_{\mathrm{max}}
		\quad\text{for a.e. }x\in\Omega,\quad0\leq t\leq1.
	\end{equation}
	If $a_0\geq \frac{a_{\mathrm{max}}}{2}$, then
	\begin{equation}\label{eq:energyinequality}
		\boxed{
			E_h(u_h^{n+1}) + \tau\mathcal D_h(W_h^{n+1})
			+(a_0-\frac{a_{\mathrm{max}}}{2})\|\delta u_h\|^2 
			\leq E_h(u_h^n).
		}
	\end{equation}
\end{thrm}

\begin{proof}
	In the row-zero convention, the coefficient arrays of EIN1 are
	\begin{equation*}
		A^I=\begin{pmatrix}0&0\\0&1\end{pmatrix},\quad
		A^E=\begin{pmatrix}0&0\\1&0\end{pmatrix}.
	\end{equation*}
	Thus $B=C=T=(1)$, and \eqref{eq:stageHQ} gives $H_0=Q=(1)$. The conditions of \cref{thm:stagecriterion} hold with $\eta=q=1$ and $d_h^{(1)}=\delta u_h$. Definition~\eqref{eq:stagepotentials} gives
	\begin{equation}\label{eq:ein1-stagepotential}
		z_h^{(1)}=\phi_h^n+a_0\delta u_h=W_h^{n+1}.
	\end{equation}
	Substitution into \eqref{eq:generalenergy} proves \eqref{eq:energyinequality}.
\end{proof}

\begin{rmrk}
	Estimate~\eqref{eq:energyinequality} has the form of \eqref{eq:single-step-stability}, with $W_h^{n+1}$ as the EIN approximation of the Kirchhoff transform.
\end{rmrk}

\begin{rmrk}[Coefficient bounds]
	A fixed global bound $0\leq a\leq a_{\mathrm{max}}$ makes the choice $a_0\geq \frac{a_{\mathrm{max}}}{2}$ independent of any state-range estimate. For an unbounded coefficient, such as a porous-medium coefficient, the bound must be valid along the old-to-new state segments throughout each element. The energy estimate does not by itself establish that range.
\end{rmrk}

\begin{rmrk}[Accumulated dissipation and initial-data requirements]
	For fixed $a_0 \geq \frac{a_{\mathrm{max}}}{2}$ and a common bound $a_{\mathrm{max}}$, summing the estimate gives
	\begin{equation}\label{eq:terminalkirchhoff}
		E_h(u_h^N) + \sum_{n=0}^{N-1}\tau\mathcal D_h(W_h^{n+1}) + (a_0-\frac{a_{\mathrm{max}}}{2})\sum_{n=0}^{N-1}\|u_h^{n+1}-u_h^n\|^2\leq E_h(u_h^0).
	\end{equation}
	The energy dissipation in \eqref{eq:terminalkirchhoff} corresponds to \eqref{eq:space-time-stability}. For variable steps, replace each $\tau$ by $\tau_n$. If $\mathcal W_h(t)=W_h^{n+1}$ on $(t_n,t_{n+1}]$, the first sum equals $\int_0^T\mathcal D_h(\mathcal W_h(t))\,dt$.
\end{rmrk}

\subsection{The standard EIN2 method}

For the standard EIN2 formula of \cite{10.1007/s11425-018-9524-x}, the stages are
\begin{equation}\label{eq:ein2stages}
	\begin{aligned}
		\frac{u_h^{(1)}-u_h^n}{\tau/2} & = L_h[\phi_h^n+a_0(u_h^{(1)}-u_h^n)],\\
		\frac{u_h^{n+1}-u_h^n}{\tau} & = L_h\left[\phi_h^{(1)}-a_0u_h^{(1)} + \frac{a_0}{2}(u_h^n+u_h^{n+1})\right].
	\end{aligned}
\end{equation}
Each solve uses the same invertible operator $I-\frac{\tau a_0}{2}L_h$.

In the row-zero convention of \eqref{eq:generalimex}, the coefficient arrays are
\begin{equation*}
	A^I=\begin{pmatrix}0&0&0\\0&\frac{1}{2}&0\\\frac{1}{2}&0&\frac{1}{2}\end{pmatrix},\quad 
	A^E=\begin{pmatrix}0&0&0\\\frac{1}{2}&0&0\\0&1&0\end{pmatrix}.
\end{equation*}
Their common abscissae are $(0,\frac{1}{2},1)^T$. 

\begin{thrm}[Energy stability of standard EIN2]\label{thm:ein2potential}
	Assume $0\leq a\leq a_{\mathrm{max}}$ on the two adjacent pointwise stage segments $[u_h^n(x),u_h^{(1)}(x)]$ and $[u_h^{(1)}(x),u_h^{n+1}(x)]$ for a.e. $x\in\Omega$. If 
	\begin{equation}\label{eq:originalein2sharp}
		a_0\geq\frac{2 a_{\mathrm{max}}}{3-\sqrt2} \approx 1.26120 a_{\mathrm{max}},
	\end{equation}
	then
	\begin{equation}\label{eq:originalein2energy}
		E_h(u_h^{n+1})
		+\frac{3-\sqrt2}{2}\tau\sum_{i=1}^2\mathcal D_h(z_h^{(i)})
		+\left(\frac{3-\sqrt2}{4}a_0-\frac{a_{\mathrm{max}}}{2}\right)
		\sum_{i=1}^2\|d_h^{(i)}\|^2\leq E_h(u_h^n).
	\end{equation}
\end{thrm}

\begin{proof}
	Substitution of the arrays into \eqref{eq:stageHQ} gives
	\begin{equation*}
		H_0=\begin{pmatrix}2&0\\1&1\end{pmatrix},\quad
		Q=\begin{pmatrix}1&0\\-\frac{1}{2}&\frac{1}{2}\end{pmatrix}.
	\end{equation*}
	The eigenvalues of $\sym H_0$ are $\frac{3\pm\sqrt2}{2}$, and those of $\sym Q$ are $\frac{3\pm\sqrt2}{4}$. Hence \eqref{eq:HQcondition} holds with $\eta=\frac{3-\sqrt2}{2}$ and $q=\frac{3-\sqrt2}{4}$. \Cref{thm:stagecriterion} proves \eqref{eq:originalein2energy}.
\end{proof}

Estimate~\eqref{eq:originalein2energy} has the one-step form of \eqref{eq:single-step-stability}. Summing over time gives \eqref{eq:generalenergysum} with $\eta=\frac{3-\sqrt2}{2}$ and $q=\frac{3-\sqrt2}{4}$, and hence the space-time form of \eqref{eq:space-time-stability}.

\subsection{The EIN2-LWW method}

Use the Lobatto-type IERK2-2 family of \cite{10.4208/csiam-am.SO-2025-0095} with $c=\frac{1}{\sqrt2}$ and choose $\gamma=\frac{1+\sqrt2}{2}$. Its full arrays are
\begin{equation}\label{eq:betterein2tableau}
	A^I = \begin{pmatrix}
		0&0&0\\-\frac{1}{2}&\gamma&0\\2-\gamma&-1&\gamma
	\end{pmatrix},\quad
	A^E = \begin{pmatrix}
		0&0&0\\c&0&0\\1-c&c&0
	\end{pmatrix}.
\end{equation}
Both row sums are $(0,c,1)^T$. The two implicit solves use the same operator $I + \gamma\tau a_0 K_h$.

\begin{thrm}[Energy stability of EIN2-LWW]\label{thm:betterein2}
	For \eqref{eq:betterein2tableau}, if $0\leq a\leq a_{\mathrm{max}}$ on the two adjacent pointwise stage segments for a.e. $x\in\Omega$ and $a_0\geq \frac{a_{\mathrm{max}}}{2}$, then
	\begin{equation}\label{eq:betterein2energy}
		E_h(u_h^{n+1})+\tau\sum_{i=1}^2\mathcal D_h(z_h^{(i)})
		+(a_0-\frac{a_{\mathrm{max}}}{2})\sum_{i=1}^2\|d_h^{(i)}\|^2\leq E_h(u_h^n).
	\end{equation}
\end{thrm}

\begin{proof}
	Direct calculation yields
	\begin{equation}\label{eq:betterein2HQ}
		H_0=\begin{pmatrix}\sqrt2&0\\2\sqrt2-2&\sqrt2\end{pmatrix},
		\quad Q=\begin{pmatrix}1+c&0\\-\sqrt2&1+c\end{pmatrix}.
	\end{equation}
	The eigenvalues of $\sym H_0$ are $1,2\sqrt2-1$, and those of $\sym Q$ are $1,1+\sqrt2$. Therefore $\eta=q=1$. \Cref{thm:stagecriterion} now gives \eqref{eq:betterein2energy}.
\end{proof}

\begin{rmrk}
	Estimate~\eqref{eq:betterein2energy} is a one-step counterpart of \eqref{eq:single-step-stability}. Its sum over time is \eqref{eq:generalenergysum} with $\eta=q=1$, corresponding to \eqref{eq:space-time-stability}.
\end{rmrk}

\subsection{The EIN3-FTY method}

The four-stage third-order IMEX-RK scheme of \cite[Eqn.~(5.12)--(5.14)]{10.1090/mcom/3950} admits the following exact rational realization of its displayed coefficients by our calculation, 
\begin{subequations}\label{eq:fty}
	\begin{equation}
		A^I_{\rm FTY}=\begin{pmatrix}
			0&0&0&0&0\\
			0&\frac{3}{5}&0&0&0\\
			0&\frac{15}{32}&\frac{33}{32}&0&0\\
			0&\frac{2}{5}&-\frac{357}{640}&\frac{709}{640}&0\\
			0&\frac{2825}{756}&-\frac{232}{297}&-\frac{6400}{231}&\frac{103}{4}
		\end{pmatrix},
	\end{equation}
	\begin{equation}
		A^E_{\rm FTY}=\begin{pmatrix}
			0&0&0&0&0\\
			\frac{3}{5}&0&0&0&0\\
			\frac{51}{64}&\frac{45}{64}&0&0&0\\
			\frac{2}{5}&\frac{4841}{11520}&\frac{299}{2304}&0&0\\
			\frac{103}{342}&\frac{125}{378}&-\frac{26}{297}&\frac{2000}{4389}&0
		\end{pmatrix}.
	\end{equation}
\end{subequations}
The shared abscissae are $c=(0,\frac{3}{5},\frac{3}{2},\frac{19}{20},1)^T$. Abscissae larger than one cause no restriction for this autonomous equation. 

\begin{thrm}[Energy stability of EIN3-FTY]\label{thm:ein3fty}
	Apply \eqref{eq:generalimex} with \eqref{eq:fty}. If $0\leq a\leq a_{\mathrm{max}}$ on all four adjacent pointwise stage segments for a.e. $x\in\Omega$ and $a_0 \geq \frac{a_{\mathrm{max}}}{2}$, then
	\begin{equation}\label{eq:ftyenergy}
		E_h(u_h^{n+1}) + \frac{\tau}{16} \sum_{i=1}^4 \mathcal D_h(z_h^{(i)}) + (a_0 - \frac{a_{\mathrm{max}}}{2}) \sum_{i=1}^4 \|d_h^{(i)}\|^2 
		\leq E_h(u_h^n).
	\end{equation}
\end{thrm}

\begin{proof}
	Direct substitution into \eqref{eq:stageHQ} gives
	\begin{equation*}
		H_0=\begin{pmatrix}
			\frac{5}{3} & 0 & 0 & 0\\
			-\frac{7}{15} & \frac{64}{45} & 0 & 0\\
			\frac{91487}{22425} & \frac{208576}{67275} & \frac{2304}{299} & 0\\
			\frac{11464607}{5175000} & \frac{54586297}{31050000} & \frac{169043}{46000} & \frac{4389}{2000}
		\end{pmatrix},
	\end{equation*}
	\begin{equation*}
		Q=\begin{pmatrix}
			1 & 0 & 0 & 0\\
			0 & \frac{22}{15} & 0 & 0\\
			0 & -\frac{33887}{22425} & \frac{12762}{1495} & 0\\
			0 & -\frac{338171431}{41400000} & -\frac{3359543}{920000} & \frac{452067}{8000}
		\end{pmatrix},
	\end{equation*}
  with smallest eigenvalues of the symmetric parts being
  \begin{equation*}
    \begin{aligned}
      & \lambda_{\min}(\sym H_0) \approx 0.0872304396420106, \\
      & \lambda_{\min}(\sym Q) = 1.
    \end{aligned}
  \end{equation*}
	For the bounds used here, Sylvester's criterion gives $\sym H_0-\frac1{16}I\succ0$. The first row and column of $\sym Q-I$ vanish, and its lower-right $3\times3$ block is positive definite by the same criterion. Thus \eqref{eq:HQcondition} holds with $\eta=\frac{1}{16}$ and $q=1$, and \cref{thm:stagecriterion} proves \eqref{eq:ftyenergy}.
\end{proof}

\begin{rmrk}
	The implicit diagonal factors are $\frac{3}{5},\frac{33}{32},\frac{709}{640},\frac{103}{4}$. Thus this small $a_0$ bound comes with four distinct implicit operators and a large last diagonal factor. At $a_0=\frac{a_{\mathrm{max}}}{2}$, the last operator has coefficient $\frac{103}{8}a_{\mathrm{max}}$ in front of $\tau K_h$. This is relevant when comparing stabilization strength, although it does not alter third-order consistency.
\end{rmrk}

\subsection{The EIN3-LWW method}

Use IERK3-2 of \cite[Section~4.2]{10.4208/csiam-am.SO-2025-0095}, with parameter $a^I_{43}=-\frac{3}{5}$ in their one-based stage convention. Including the old value as row zero, its exact arrays are
\begin{subequations}\label{eq:lww}
	\begin{equation}
		A^I_{\rm LWW}=\begin{pmatrix}
			0&0&0&0&0\\
			\frac{2}{25}&\frac{18}{25}&0&0&0\\
			\frac{3}{8}&\frac{61}{200}&\frac{18}{25}&0&0\\
			\frac{1607}{12600}&\frac{12001}{12600}&-\frac{3}{5}&\frac{18}{25}&0\\
			\frac{1030769}{2877000}&\frac{276523}{1233000}&-\frac{196127}{1078875}&-\frac{2068}{17125}&\frac{18}{25}
		\end{pmatrix},
	\end{equation}
	\begin{equation}
		A^E_{\rm LWW}=\begin{pmatrix}
			0&0&0&0&0\\
			\frac{4}{5}&0&0&0&0\\
			\frac{3}{5}&\frac{4}{5}&0&0&0\\
			\frac{10111}{10080}&-\frac{6079}{10080}&\frac{4}{5}&0&0\\
			\frac{313}{840}&\frac{131}{360}&-\frac{169}{315}&\frac{4}{5}&0
		\end{pmatrix}.
	\end{equation}
\end{subequations}
Here $c=(0,\frac{4}{5},\frac{7}{5},\frac{6}{5},1)^T$. Unlike the FTY choice, all four solves reuse the single operator $I + (\frac{18}{25}) \tau a_0 K_h$ and may have computational benefits.

\begin{thrm}[Energy stability of EIN3-LWW]\label{thm:ein3lww}
	For \eqref{eq:lww}, assume $0\leq a\leq a_{\mathrm{max}}$ on the four adjacent pointwise state segments for a.e. $x\in\Omega$. If $a_0 \geq a_{\mathrm{max}}$, then
	\begin{equation}\label{eq:lwwenergy}
		E_h(u_h^{n+1}) + \frac{\tau}{8} \sum_{i=1}^4 \mathcal D_h(z_h^{(i)}) + \frac{a_0-a_{\mathrm{max}}}{2} \sum_{i=1}^4 \|d_h^{(i)}\|^2 
		\leq E_h(u_h^n).
	\end{equation}
\end{thrm}

\begin{proof}
	Direct substitution into \eqref{eq:stageHQ} gives
	\begin{equation*}
		H_0=\begin{pmatrix}
			\frac{5}{4} & 0 & 0 & 0\\
			\frac{5}{16} & \frac{5}{4} & 0 & 0\\
			-\frac{3515}{43008} & \frac{70715}{32256} & \frac{5}{4} & 0\\
			\frac{5102845}{10838016} & \frac{17489795}{8128512} & \frac{2105}{1008} & \frac{5}{4}
		\end{pmatrix},
	\end{equation*}
	\begin{equation*}
		Q=\begin{pmatrix}
			\frac{9}{10} & 0 & 0 & 0\\
			-\frac{63}{160} & \frac{9}{10} & 0 & 0\\
			-\frac{47433}{143360} & -\frac{1537}{8960} & \frac{9}{10} & 0\\
			-\frac{3663025711}{8248934400} & -\frac{513442597}{1546675200} & \frac{67633}{191800} & \frac{9}{10}
		\end{pmatrix},
	\end{equation*}
  with smallest eigenvalues of the the symmetric parts being 
  \begin{equation*}
    \begin{aligned}
      & \lambda_{\min}(\sym H_0) \approx 0.136354899669695, \\
      & \lambda_{\min}(\sym Q) \approx 0.502118851290468.
    \end{aligned}
  \end{equation*}
	Sylvester's criterion gives $\sym H_0-\frac18I\succ0$ and $\sym Q-\frac12I\succ0$. Hence \eqref{eq:HQcondition} holds with $\eta=\frac{1}{8}$ and $q=\frac{1}{2}$, and \cref{thm:stagecriterion} proves \eqref{eq:lwwenergy}.
\end{proof}

\subsection{The EIN3-ARS(4,4,3) method: Failure in the framework}

The four-stage third-order EIN formula used in \cite[Eq.~(5.1)]{10.1007/s11425-018-9524-x} is the ARS(4,4,3) formula of \cite{10.1016/S0168-92749700056-1}. It differs from the two third-order formulae considered above and does not satisfy the assumptions of \cref{thm:stagecriterion}.

Indeed, including the old value as stage zero, its last diagonal implicit coefficient and last subdiagonal explicit coefficient are
\begin{equation*}
	a^I_{44} = \frac12, \quad a^E_{43} = -\frac74.
\end{equation*}
The lower triangular matrices
\begin{equation*}
	B=(a^I_{ij})_{i,j=1}^4,\quad
	C=(a^E_{i,j-1})_{i,j=1}^4
\end{equation*}
therefore satisfy
\begin{equation*}
	(H_0)_{44}=\frac{1}{C_{44}}=-\frac47,\quad
	Q_{44}=\frac{B_{44}}{C_{44}}=-\frac27,
\end{equation*}
where
\begin{equation*}
	H_0=C^{-1}T,\quad Q=(C^{-1}B-I)T+I.
\end{equation*}
A positive semidefinite symmetric matrix must have nonnegative diagonal entries. Consequently, neither $\sym H_0\succeq0$ nor $\sym Q\succeq qI$ with $q>0$ holds for this formula. These coefficient matrices are independent of $a_0$, so increasing the stabilization parameter does not make the formula satisfy the present criterion.

This observation concerns the sufficient conditions of \cref{thm:stagecriterion}. It does not imply that the EIN-ARS(4,4,3) formula is energy unstable.

\section{Generalizations}\label{sec:generalizations}

\subsection{Diffusion systems with a scalar energy potential}

\subsubsection{A fixed-metric potential structure}

Let $\mathbf{u}: \Omega \mapsto \mathcal U\subset\mathbb R^m$, where $\mathcal U$ is open and convex, and consider
\begin{equation}
	\mathbf{u}_t=\sum_{\alpha=1}^d\partial_{x_\alpha} \left( A(\mathbf{u})\partial_{x_\alpha} \mathbf{u}\right).
\end{equation}
We impose homogeneous Dirichlet boundary conditions as in \cref{sec:space}. Suppose a convex scalar energy potential $\Psi\in C^2(\mathcal U)$ and a fixed symmetric positive definite matrix $H\in\mathbb R^{m\times m}$ satisfy
\begin{equation}\label{eq:systemhessian}
	H A(\mathbf{u}) = \nabla^2\Psi(\mathbf{u}),\quad
	0 \preceq H^{-\frac{1}{2}} \nabla^2 \Psi(\mathbf{u}) H^{-\frac{1}{2}} \preceq A_{\mathrm{max}} I_m.
\end{equation}
Here $A_{\mathrm{max}}$ is an eigenvalue bound. With a reference state $\mathbf{u}_*$, set
\begin{equation}\label{eq:systemPhi}
	\Phi(\mathbf{u}) = H^{-1}\{\nabla\Psi(\mathbf{u}) - \nabla\Psi(\mathbf{u}_*)\},\quad
	D\Phi(\mathbf{u}) = A(\mathbf{u}).
\end{equation}
The Kirchhoff transform $\Phi$ is now vector-valued, whereas the energy potential $\Psi$ remains a scalar. Their relation is $H \Phi(\mathbf{u}) = \nabla\Psi(\mathbf{u}) - \nabla\Psi(\mathbf{u}_*)$; in the Euclidean metric $H = I_m$, this is the vector counterpart of $\psi' = \phi$. In the proof below we take $\mathbf{u}_*=0$, so that both $\mathbf{u}$ and $\Phi(\mathbf{u})$ have homogeneous Dirichlet data. For a constant nonzero boundary state $\mathbf{u}_*$, apply the proof to $\mathbf{v} = \mathbf{u} - \mathbf{u}_*$, $\widetilde\Psi(\mathbf{v}) = \Psi(\mathbf{u}_* + \mathbf{v}) - \Psi(\mathbf{u}_*) - \nabla\Psi(\mathbf{u}_*)\cdot \mathbf{v}$, and $\widetilde\Phi(\mathbf{v}) = \Phi(\mathbf{u}_* + \mathbf{v})$. The differential identity $D \Phi = A$ gives $\mathbf{u}_t = \Delta\Phi(\mathbf{u})$, and the normalized continuous energy satisfies
\begin{equation*}
	\frac{\d}{\d t} \int_\Omega \{\Psi(\mathbf{u}) - \Psi(\mathbf{u}_*) - \nabla\Psi(\mathbf{u}_*) \cdot (\mathbf{u} - \mathbf{u}_*)\} \, \d \mathbf{x}
	=-\sum_{\alpha=1}^d \int_\Omega (\partial_{x_\alpha}\Phi)^T H \partial_{x_\alpha}\Phi\,\d \mathbf{x}
\end{equation*}
under the boundary conditions of \cref{sec:space}. Thus $H^{-1}$ is a constant Onsager mobility and $A = H^{-1} \nabla^2\Psi$. The Euclidean Hessian case is $H = I_m$; other fixed metrics allow nonsymmetric diffusion matrices.

For a $C^1$ matrix field, let $S(\mathbf{u}) = H A(\mathbf{u})$. On a convex state domain the compatibility conditions for a scalar energy potential are
\begin{equation}\label{eq:hessiancompatibility}
	S_{ij} = S_{ji},\quad
	\partial_{u_k} S_{ij}=\partial_{u_j} S_{ik} \quad i,j,k=1,\ldots,m.
\end{equation}
The differential condition integrates each row to a component of a vector field, and symmetry makes that vector field a gradient. Nonnegative definiteness makes its potential convex. A fixed symmetrizer alone does not imply the differential condition. The structure in \eqref{eq:systemhessian} is sufficient for the present proof; it is not asserted to be necessary for every energy estimate for a diffusion system.

\subsubsection{Projected vector Kirchhoff transform and energy estimate}

Use $(V_h)^m$ with the exact inner product
\begin{equation*}
	(\mathbf{v}, \mathbf{w})_H = \int_\Omega \mathbf{v}^T H \mathbf{w}\,\d \mathbf{x},\quad\|\mathbf{v}\|_H^2=(\mathbf{v}, \mathbf{v})_H.
\end{equation*}
The componentwise orthogonal projection $\PV$ is also orthogonal in this inner product because $H$ is fixed. Let $K_h$ act identically on all components. It is self-adjoint and positive semidefinite in $(\cdot,\cdot)_H$. Define
\begin{equation*}
	\mathcal D_{h,H} (\mathbf{v}_h) \coloneqq (K_h \mathbf{v}_h, \mathbf{v}_h)_H \geq 0.
\end{equation*}
The normalized energy or free energy is
\begin{equation}\label{eq:systemenergy}
	E(\mathbf{v}) \coloneqq \int_\Omega \{\Psi(\mathbf{v}) - \Psi(\mathbf{u}_*) - \nabla \Psi(\mathbf{u}_*) \cdot (\mathbf{v} - \mathbf{u}_*)\}\,\d \mathbf{x} \geq 0.
\end{equation}
Write $E_h:=E|_{(V_h)^m}$. Its Fr\'echet derivative satisfies
\begin{equation*}
	\langle D E_h(\mathbf{v}_h), \mathbf{w}_h\rangle_{((V_h)^m)',(V_h)^m} = (\Phi(\mathbf{v}_h), \mathbf{w}_h)_H = (\PV\Phi(\mathbf{v}_h), \mathbf{w}_h)_H.
\end{equation*}
Use $\Phi_h^n \coloneqq \PV\Phi(\mathbf{u}_h^n)$ and $\Phi_h^{(j)} \coloneqq \PV\Phi(\mathbf{u}_h^{(j)})$ in the componentwise stage equations.
The corresponding stage vectors are
\begin{equation*}
	\begin{aligned}
		[\mathbf{u}_h^\ast]&=(\mathbf{u}_h^{(1)},\ldots,\mathbf{u}_h^{(s)})^T,\\
		[\mathbf{u}_h]&=(\mathbf{u}_h^{(0)},\ldots,\mathbf{u}_h^{(s-1)})^T,\\
		[\Phi_h]&=(\Phi_h^{(0)},\ldots,\Phi_h^{(s-1)})^T.
	\end{aligned}
\end{equation*}
With $[\mathbf{d}_h]=(\mathbf{d}_h^{(1)},\ldots,\mathbf{d}_h^{(s)})^T$ and $\mathbf{d}_h^{(i)}=\mathbf{u}_h^{(i)}-\mathbf{u}_h^{(i-1)}$, the identities
\begin{equation*}
    [\mathbf{u}_h^\ast]=\bm e\mathbf{u}_h^n+T[\mathbf{d}_h],\quad
    [\mathbf{u}_h]=\bm e\mathbf{u}_h^n+(T-I)[\mathbf{d}_h]
\end{equation*}
give the same elimination as in the scalar case. Define the vector stage potentials by
\begin{equation}\label{eq:vectorstagepotentials}
    [\mathbf{z}_h] \coloneqq H_0^{-1}\left([\Phi_h]+a_0Q[\mathbf{d}_h]\right), \quad
    [\mathbf{z}_h] = (\mathbf{z}_h^{(1)},\ldots,\mathbf{z}_h^{(s)})^T.
\end{equation}
Here $H_0^{-1}$ and $Q$ act only on the temporal stage index, with the same action on every physical component.

\begin{thrm}[Fixed-metric systems]\label{thm:systems}
	Assume \eqref{eq:systemhessian} on the pointwise state segments required by \cref{thm:stagecriterion}. Suppose that the temporal arrays satisfy \eqref{eq:HQcondition}, and retain the scalar stabilizer $a_0 I_m$. If $a_0 \geq \frac{A_{\mathrm{max}}}{2q}$, then
	\begin{equation}\label{eq:systemstageenergy}
		E_h(\mathbf{u}_h^{n+1}) + \tau \eta \sum_i \mathcal D_{h,H}(\mathbf{z}_h^{(i)})
		+(a_0 q - \frac{A_{\mathrm{max}}}{2}) \sum_i \|\mathbf{d}_h^{(i)}\|_H^2
		\leq E_h(\mathbf{u}_h^n),
	\end{equation}
	where the vector stage potentials are defined by \eqref{eq:vectorstagepotentials}. The sufficient bounds are $\frac{A_{\mathrm{max}}}{2}$ for EIN1, $\frac{2 A_{\mathrm{max}}}{3-\sqrt{2}}$ for the standard EIN2 estimate, $\frac{A_{\mathrm{max}}}{2}$ for EIN2-LWW and EIN3-FTY, and $A_{\mathrm{max}}$ for EIN3-LWW. 
\end{thrm}

\begin{proof}
	Apply pointwise Taylor expansion to the energy function in \eqref{eq:systemenergy}. The Hessian bound gives
	\begin{equation*}
		\Psi(\mathbf{u}+\mathbf{d}) - \Psi(\mathbf{u}) - \nabla\Psi(\mathbf{u}_*) \cdot \mathbf{d}
		\leq \Phi(\mathbf{u})^T H \mathbf{d} + \frac{A_{\mathrm{max}}}{2}\mathbf{d}^T H \mathbf{d}.
	\end{equation*}
	Integration and orthogonality of $\PV$ yield
	\begin{equation*}
		E_h(\mathbf{u}_h^{(i-1)} + \mathbf{d}_h^{(i)}) - E_h(\mathbf{u}_h^{(i-1)})
		\leq (\Phi_h^{(i-1)}, \mathbf{d}_h^{(i)})_H + \frac{A_{\mathrm{max}}}{2}\|\mathbf{d}_h^{(i)}\|_H^2.
	\end{equation*}
	The stage identities \eqref{eq:stageelimination} hold componentwise. The matrices $(K_h \mathbf{z}_h^{(i)}, \mathbf{z}_h^{(j)})_H$ and $(\mathbf{d}_h^{(i)}, \mathbf{d}_h^{(j)})_H$ are symmetric positive semidefinite. Applying \eqref{eq:HQcondition} as in the scalar proof gives \eqref{eq:systemstageenergy} and its time-summed version. 
\end{proof}

Estimate~\eqref{eq:systemstageenergy} is the fixed-metric counterpart of the one-step estimate \eqref{eq:single-step-stability}. For variable steps $\tau_n > 0$, its sum over time gives
\begin{equation}
	E_h(\mathbf{u}_h^N) + \sum_{n=0}^{N-1} \left(\tau_n\eta\sum_i\mathcal D_{h,H}(\mathbf{z}_h^{n,(i)}) + (a_0 q - \frac{A_{\mathrm{max}}}{2}) \sum_i \|\mathbf{d}_h^{n,(i)}\|_H^2 \right)
	\leq E_h(\mathbf{u}_h^0),
\end{equation}
which has the space-time form of \eqref{eq:space-time-stability} in the $H$ inner product. For EIN1, $\mathbf{z}_h^{(1)} = \mathbf{W}_h^{n+1}$ gives the potential dissipation in the scalar estimate. Every implicit stage solves $I + \tau a_0 a^I_{ii} K_h$ separately on each component. The fixed metric therefore preserves the linear implicit solves of the scalar method.

\subsection{Quadrature and other boundary conditions}\label{rem:quadrature}

The preceding energy estimates also apply when positive-weight quadrature is used for the nonlinear integrals, with a consistent definition of the energy and projected Kirchhoff transform. For example, retain the exact $L^2$ inner product and the Dirichlet operator $K_h$, and let $Q_\Omega$ be a positive-weight volume rule satisfying
\begin{equation*}
	Q_\Omega(v_h \cdot w_h) = (v_h, w_h) \quad v_h, w_h\in V_h.
\end{equation*}
Define the quadrature energy and projected Kirchhoff transform by
\begin{equation*}
	E_{h,Q}(v_h) \coloneqq Q_\Omega(\psi(v_h)), \quad
	(\Pi_{V_h}^{Q}\phi(v_h),w_h)
		=Q_\Omega(\phi(v_h)w_h).
\end{equation*}
Then
\begin{equation*}
\begin{aligned}
	\langle DE_{h,Q}(v_h),w_h\rangle_{V_h',V_h}
		&=(\Pi_{V_h}^{Q}\phi(v_h),w_h),\\
	E_{h,Q}(v_h+d_h)-E_{h,Q}(v_h)
		&\leq(\Pi_{V_h}^{Q}\phi(v_h),d_h)+\frac{a_{\mathrm{max}}}{2}\|d_h\|^2.
\end{aligned}
\end{equation*}
provided $0\leq a\leq a_{\mathrm{max}}$ on the sampled state segments. All preceding proofs therefore hold with $E_{h,Q}$ and $\Pi_{V_h}^{Q}$, with the same constants. No polynomial exactness for the nonlinear energy $\psi(v_h)$ is needed. Exactness for products in $V_h$ suffices for this version.

Alternatively, one may use the quadrature inner product throughout, provided it is positive definite on $V_h$, the projection is orthogonal in that inner product, and the discrete Laplacian is self-adjoint and positive semidefinite in it. Positive weights alone do not ensure these operator properties. Exactness for the polynomial products in the mass and elliptic forms is a sufficient way to preserve them; otherwise they must be verified directly. The quadrature energy, rather than the exact nonlinear integral, is then dissipated. The same statements hold for the fixed-metric systems.

Homogeneous Neumann and periodic boundary conditions are also covered by choosing the corresponding self-adjoint positive semidefinite discrete Laplacian. The stage identities and energy estimates remain valid without additional constraints on the stages.

\section{Applications}\label{sec:applications}

We first identify several classical diffusion models with the required potential. The coefficient bounds below must hold on the pointwise segments specified in the corresponding energy theorem. They do not assert preservation of an admissible state range by the time discretization.

\subsection{Gas flow in a porous medium}

The porous-medium equation arises from Darcy flow of a compressible gas; see the classical derivation of Muskat \cite{10.1063/1.1745233}. With constant porosity $\varphi$, permeability $k$, viscosity $\mu$, and the polytropic law $p=c\rho^\gamma$, mass conservation gives
\begin{equation*}
	\varphi\rho_t+\nabla\cdot(\rho v)=0,\quad
	v=-\frac{k}{\mu}\nabla p,
\end{equation*}
and therefore
\begin{equation}
	\rho_t=\frac{kc\gamma}{\varphi\mu}
		\nabla\cdot(\rho^\gamma\nabla\rho)
		=D\Delta(\rho^m),\quad
	m=\gamma+1>1,\quad D=\frac{kc\gamma}{\varphi\mu m}.
\end{equation}
For $\rho\geq0$ the Kirchhoff transform and energy are
\begin{equation*}
	\phi(\rho)=D\rho^m,\quad
	\psi(\rho)=\frac{D}{m+1}\rho^{m+1},\quad
	a(\rho)=Dm\rho^{m-1}.
\end{equation*}
On segments with $0\leq\rho\leq R$, take $a_{\mathrm{max}}=D m R^{m-1}$. The isothermal ideal-gas case has $\gamma=1$, hence $m=2$.

\subsection{Horizontal unsaturated flow}

In the absence of gravity, the moisture-content form of the Richards model \cite{10.1063/1.1745010} is
\begin{equation}
	\theta_t=\nabla\cdot\{K(\theta)\nabla p(\theta)\}
	=\nabla\cdot\{D(\theta)\nabla\theta\},\quad
	D(\theta)=K(\theta)p'(\theta)\geq0.
\end{equation}
Here $K$ is the hydraulic conductivity and $p$ is the pressure head, increasing with moisture content on the constitutive branch under consideration. Thus
\begin{equation*}
	\phi(\theta)=\int_{\theta_*}^{\theta}D(s)\,ds,\quad
	\psi(\theta)=\int_{\theta_*}^{\theta}\phi(s)\,ds.
\end{equation*}
Take $\theta_*$ as a constant Dirichlet datum and apply the scheme to $\theta-\theta_*$, with the corresponding shifted potential. The energy estimates apply with an upper bound for $D$ on the relevant moisture range. A compact range inside the constitutive interval supplies such a bound when $D$ is continuous. This example concerns horizontal flow, or the diffusion part of the Richards equation; the gravity term in the full model is a separate transport term. EIN treatment of the full Richards model is studied in \cite{10.1016/j.compgeo.2026.108315}.

\subsection{Electron heat conduction in a plasma}

The Spitzer--H\"arm law \cite{10.1103/PhysRev.89.977} gives a conductivity proportional to $T^{\frac{5}{2}}$ for a collisional, unmagnetized, fully ionized plasma, with the Coulomb logarithm treated as constant. At fixed density and heat capacity, the isotropic conduction model becomes
\begin{equation}
	T_t=\chi\nabla\cdot(T^{\frac{5}{2}}\nabla T),\quad \chi>0.
\end{equation}
For nonnegative temperature,
\begin{equation*}
	\phi(T)=\frac{2\chi}{7}T^{\frac{7}{2}},\quad
	\psi(T)=\frac{4\chi}{63}T^{\frac{9}{2}},\quad
	a(T)=\chi T^{\frac{5}{2}}.
\end{equation*}
Thus $a_{\mathrm{max}}=\chi R^{\frac{5}{2}}$ on segments with $0\leq T\leq R$. The convex functional used here is a mathematical energy for the diffusion equation. It need not coincide with the thermodynamic energy of the plasma.

\subsection{Multicomponent diffusion with constant Onsager mobility}

Onsager's reciprocal relations \cite{10.1103/PhysRev.37.405,10.1103/PhysRev.38.2265} provide the constitutive law for coupled irreversible transport. For isothermal diffusion of $m$ solutes and a solvent in a closed mixture, write $u_0=1-\sum_{i=1}^m u_i$, where the components of \(\mathbf{u} \in \R^m\) represent the solute fractions, and use the dimensionless ideal-mixture free-energy density
\begin{equation}
	f(\mathbf{u}) = \sum_{i=1}^m u_i\log u_i+u_0\log u_0,
	\quad u_i > 0,\quad u_0 > 0.
\end{equation}
Its derivatives are the solute--solvent chemical-potential differences $\mu_i=\log(\frac{u_i}{u_0})$; see \cite{DGM1962} for the thermodynamic formulation of multicomponent diffusion. With a constant symmetric positive definite Onsager mobility $M$, the fluxes and conservation laws are
\begin{equation}
	J_i = -\sum_{j=1}^m M_{ij}\nabla\mu_j,\quad
	\mathbf{u}_t = -\nabla\cdot J = M\Delta\nabla f(\mathbf{u}).
\end{equation}
The solvent flux is $J_0 = -\sum_i J_i$. This is the constant-mobility specialization of the Onsager diffusion law. No composition-dependent mobility is assumed. It has
\begin{equation}
	A(\mathbf{u}) = M\left\{\diag(u_1^{-1},\ldots,u_m^{-1}) + u_0^{-1}(1)_{i,j=1}^m\right\}, \quad
	H=M^{-1},\quad \Psi = f.
\end{equation}
The vector Kirchhoff transform is $M\nabla f$, up to the constant normalization in \eqref{eq:systemPhi}. The diffusion matrix is generally nonsymmetric, while $H A=\nabla^2 f$. If every required state segment satisfies $u_i\geq\delta>0$ for $i=0,\ldots,m$, then
\begin{equation*}
	0\prec H^{-\frac{1}{2}}\nabla^2 f(\mathbf{u}) H^{-\frac{1}{2}}
	\preceq \frac{(m+1)\lambda_{\max}(M)}{\delta} I_m.
\end{equation*}
\Cref{thm:systems} therefore applies with the displayed upper bound as $A_{\mathrm{max}}$. Choose a constant interior equilibrium $\mathbf{u}_*$ as boundary data and apply the analysis to $\mathbf{u}-\mathbf{u}_*$, which has homogeneous Dirichlet data. This restriction excludes the singular concentration limits; it is an assumption on the stage range.

\subsection{Regularized vector fast diffusion}

For $0<p<1$ and $\varepsilon>0$, set
\begin{equation}\label{eq:regfast}
	\mathbf{u}_t=\Delta\left((\varepsilon^2+|\mathbf{u}|^2)^{\frac{p-1}{2}}\mathbf{u}\right),\quad
	\Psi_\varepsilon(\mathbf{u})=\frac{(\varepsilon^2+|\mathbf{u}|^2)^{\frac{p+1}{2}}}{p+1},\quad
	\mathbf{u} \in \R^m. 
\end{equation}
This is a smooth regularization of vector fast diffusion, specified here to give a globally valid example. With $\rho=|\mathbf{u}|$ and $s=\varepsilon^2+\rho^2$, its Hessian is
\begin{equation*}
	 A(\mathbf{u})=s^{\frac{p-1}{2}}I_m+(p-1)s^{\frac{p-3}{2}}\mathbf{u}\mathbf{u}^T.
\end{equation*}
The tangential eigenvalue is $s^{\frac{p-1}{2}}$ and the radial eigenvalue is $(\varepsilon^2+p\rho^2)s^{\frac{p-3}{2}}$. They are positive, the radial one does not exceed the tangential one, and both are at most $A_{\mathrm{max}}=\varepsilon^{p-1}$. Thus all energy theorems apply globally with this $A_{\mathrm{max}}$, without needing a pointwise range bound.

\subsection{Globally bounded nonlinear cross diffusion}

Two bounded examples can be constructed directly from convex potentials.
For $\kappa\geq0$, take
\begin{equation}\label{eq:saturationpsi}
	\Psi(\mathbf{u}) = \frac{1}{2}(u_1^2 + u_2^2) + \kappa\log\cosh(u_1-u_2),
	\quad H = I_2.
\end{equation}
Writing $z=u_1-u_2$ gives the coupled PDE
\begin{equation*}
	(u_1)_t=\Delta(u_1+\kappa\tanh z),\quad
	(u_2)_t=\Delta(u_2-\kappa\tanh z),
\end{equation*}
and
\begin{equation*}
	A(\mathbf{u}) = I_2 + \kappa \sech^2 z
	\begin{pmatrix}1&-1\\-1&1\end{pmatrix},\quad
	I_2\preceq A \preceq (1+2\kappa)I_2.
\end{equation*}
Hence $A_{\mathrm{max}}=1+2\kappa$ is global. This gives a smooth coupled example for all orders considered, without an auxiliary range hypothesis.

To illustrate why a Euclidean Hessian is not necessary, instead set
\begin{equation}\label{eq:nonsympsi}
	H=\diag(1,\frac{1}{2}),\quad
	\Psi(\mathbf{u}) = \frac{1}{2}(u_1^2+u_2^2)+\kappa\log\cosh(u_1+u_2).
\end{equation}
Now $z=u_1+u_2$ and
\begin{equation*}
	(u_1)_t=\Delta(u_1+\kappa\tanh z),\quad
	(u_2)_t=2\Delta(u_2+\kappa\tanh z).
\end{equation*}
For $s=\sech^2 z$,
\begin{equation*}
	 A(\mathbf{u})=\begin{pmatrix}
		1+\kappa s&\kappa s\\
		2\kappa s&2+2\kappa s
	\end{pmatrix},
	\quad H A=\nabla^2\Psi.
\end{equation*}
The diffusion matrix is generally nonsymmetric, whereas its weighted form is symmetric positive definite. Since $H^{-\frac{1}{2}}\nabla^2\Psi H^{-\frac{1}{2}}$ is a diagonal matrix with entries $1,2$ plus a positive rank-one matrix of norm $3\kappa s$, the global upper bound $A_{\mathrm{max}}=2+3\kappa$ suffices. This system satisfies the hypotheses of \cref{thm:systems}, retaining the scalar stabilizer and identical componentwise discrete Laplacians.

\section{Concluding remarks}\label{sec:summary}

We have proved nonlinear energy stability for first- to third-order EIN time discretizations of nonlinear degenerate diffusion equations with a Kirchhoff transform, summarized in \cref{tab:bounds}. The nonlinear terms are evaluated explicitly, and every implicit stage uses a linear discrete backward heat solve. Under the stated coefficient bounds, all energy estimates are unconditional with respect to the time step and require no positive lower bound on $a(\cdot) \geq 0$. The methods require neither Newton--Raphson nor Picard iterations. Their spatial discretizations assume self-adjointness and positive semidefiniteness in the exact $L^2$ inner product, together with orthogonal projection of the Kirchhoff transform in the spatial discretization. A wide range of symmetric Galerkin discretizations can be used within this setting, including but not limited to symmetric realizations of CG, mixed FEM, BR1, BR2, the method of Brezzi et al., LDG, SIPG, UWDG, DDG, HDG, WG, spectral Galerkin, and mixed VEM. 

\begin{table}[htb]
	\centering
	\caption{Sufficient stabilization bounds. The coefficient bound $a_{\mathrm{max}}$ is required on the pointwise segments specified in each result.}
	\label{tab:bounds}
	\begin{tabular}{lll}
		\toprule
		Time discretization & Sufficient lower bound on $a_0$ & Energy estimate\\
		\midrule
		EIN1 & $\frac{1}{2}a_{\mathrm{max}}$ & \eqref{eq:energyinequality}\\
		Standard EIN2 & $\frac{2}{3-\sqrt{2}}a_{\mathrm{max}}$ & \eqref{eq:originalein2energy}\\
		EIN2-LWW & $\frac{1}{2}a_{\mathrm{max}}$ & \eqref{eq:betterein2energy}\\
		EIN3-FTY & $\frac{1}{2}a_{\mathrm{max}}$ & \eqref{eq:ftyenergy}\\
		EIN3-LWW & $a_{\mathrm{max}}$ & \eqref{eq:lwwenergy}\\
		Fixed-metric systems & Same bounds with $A_{\mathrm{max}}$ & \cref{thm:systems}\\
		\bottomrule
	\end{tabular}
\end{table}

For systems, a scalar convex potential in a fixed metric allows the same proof and componentwise linear implicit solves. The constant-mobility Onsager model gives a physical example with nonsymmetric diffusion, while the regularized vector fast-diffusion and bounded cross-diffusion examples admit global coefficient bounds. Porous-medium flow, plasma heat conduction, and constitutive laws with singular limits require a bound on the relevant stage ranges to select a finite common upper bound $a_{\mathrm{max}}$ or $A_{\mathrm{max}}$, respectively.

%%-----------------------------
%%      your bibliography
%%-----------------------------
\bibliographystyle{plain}
\bibliography{refs}

@article{10.1006/jcph.1996.5572,
    title   = {A High-Order Accurate Discontinuous Finite Element Method for the Numerical Solution of the Compressible {N}avier--{S}tokes Equations},
    journal = {Journal of Computational Physics},
    volume  = {131},
    number  = {2},
    pages   = {267-279},
    year    = {1997},
    month   = {03},
    issn    = {0021-9991},
    doi     = {10.1006/jcph.1996.5572},
    author  = {Bassi, Francesco and Rebay, Stefano}
}

@book{10.1007/978-3-540-71041-7,
    author    = {Shen, Jie and Tang, Tao and Wang, Li-Lian},
    title     = {Spectral Methods},
    subtitle  = {Algorithms, Analysis and Applications},
    series    = {Springer Series in Computational Mathematics},
    volume    = {41},
    publisher = {Springer},
    address   = {Berlin, Heidelberg},
    year      = {2011},
    month     = {08},
    doi       = {10.1007/978-3-540-71041-7},
    isbn      = {978-3-540-71041-7},
    issn      = {2198-3712}
}

@article{10.1007/BF01389710,
    author  = {Brezzi, Franco and Douglas, Jr., Jim and Marini, Luisa Donatella},
    title   = {Two families of mixed finite elements for second order elliptic problems},
    journal = {Numerische Mathematik},
    volume  = {47},
    number  = {2},
    pages   = {217-235},
    year    = {1985},
    month   = {06},
    doi     = {10.1007/BF01389710},
    issn    = {0945-3245}
}

@inproceedings{10.1007/BFb0064470,
    author       = {Raviart, Pierre-Arnaud and Thomas, Jean-Marie},
    title        = {A mixed finite element method for 2-nd order elliptic problems},
    booktitle    = {Mathematical Aspects of Finite Element Methods},
    booksubtitle = {Proceedings of the Conference Held in Rome, December 10 - 12, 1975},
    editor       = {Galligani, Ilio and Magenes, Enrico},
    series       = {Lecture Notes in Mathematics},
    volume       = {606},
    publisher    = {Springer},
    address      = {Berlin, Heidelberg},
    pages        = {292-315},
    year         = {1977},
    doi          = {10.1007/BFb0064470},
    isbn         = {978-3-540-37158-8},
    issn         = {1617-9692}
}

@article{10.1007/s10915-019-01072-4,
    author  = {Wang, Haijin and Zhang, Qiang and Shu, Chi-Wang},
    title   = {Implicit--explicit local discontinuous {G}alerkin methods with generalized alternating numerical fluxes for convection--diffusion problems},
    journal = {Journal of Scientific Computing},
    volume  = {81},
    number  = {1},
    pages   = {2080-2114},
    year    = {2019},
    month   = {10},
    doi     = {10.1007/s10915-019-01072-4},
    issn    = {1573-7691}
}

@article{10.1007/s11425-018-9524-x,
    author  = {Wang, Haijin and Zhang, Qiang and Wang, Shiping and Shu, Chi-Wang},
    title   = {Local discontinuous {G}alerkin methods with explicit-implicit-null time discretizations for solving nonlinear diffusion problems},
    journal = {Science China Mathematics},
    volume  = {63},
    number  = {1},
    pages   = {183-204},
    year    = {2019},
    month   = {05},
    doi     = {10.1007/s11425-018-9524-x},
    issn    = {1869-1862}
}

@article{10.1007/s42967-023-00359-6,
    author  = {Tan, Meiqi and Cheng, Juan and Shu, Chi-Wang},
    title   = {The high-order variable-coefficient explicit-implicit-null method for diffusion and dispersion equations},
    journal = {Communications on Applied Mathematics and Computation},
    volume  = {7},
    number  = {1},
    pages   = {115-150},
    year    = {2024},
    month   = {05},
    doi     = {10.1007/s42967-023-00359-6},
    issn    = {2661-8893}
}

@inproceedings{10.1016/B978-0-12-358502-8.50009-8,
    author       = {Douglas, Jr., Jim and Dupont, Todd F.},
    title        = {Alternating-direction {G}alerkin methods on rectangles},
    booktitle    = {Numerical Solution of Partial Differential Equations--II},
    booksubtitle = {Proceedings of the Second Symposium on the Numerical Solution of Partial Differential Equations, SYNSPADE 1970, Held at the University of Maryland, College Park, Maryland, May 11--15, 1970: SYNSPADE 1970},
    chapter      = {6},
    editor       = {Hubbard, Bert},
    publisher    = {Academic Press},
    address      = {New York},
    pages        = {133-214},
    year         = {1971},
    doi          = {10.1016/B978-0-12-358502-8.50009-8},
    isbn         = {978-0-12-358502-8}
}

@article{10.1016/j.cam.2012.10.003,
    author  = {Wang, Junping and Ye, Xiu},
    title   = {A weak {G}alerkin finite element method for second-order elliptic problems},
    journal = {Journal of Computational and Applied Mathematics},
    volume  = {241},
    pages   = {103-115},
    year    = {2013},
    month   = {03},
    doi     = {10.1016/j.cam.2012.10.003},
    issn    = {1879-1778}
}

@article{10.1016/j.cnsns.2024.108094,
    author  = {Wang, Yating and Leung, Wing Tat},
    title   = {Partially explicit splitting scheme with explicit--implicit-null method for nonlinear multiscale flow problems},
    journal = {Communications in Nonlinear Science and Numerical Simulation},
    volume  = {136},
    pages   = {108094},
    year    = {2024},
    month   = {09},
    doi     = {10.1016/j.cnsns.2024.108094},
    issn    = {1007-5704}
}

@article{10.1016/j.compgeo.2026.108315,
    author  = {Wu, Xinyu and Cheng, Juan and Guo, Hui and Shu, Chi-Wang},
    title   = {An efficient high-order bound-preserving explicit--implicit-null time integration for the {R}ichards' equation in unsaturated geotechnical media},
    journal = {Computers and Geotechnics},
    volume  = {198},
    pages   = {108315},
    year    = {2026},
    month   = {10},
    doi     = {10.1016/j.compgeo.2026.108315},
    issn    = {0266-352X}
}

@article{10.1016/j.jcp.2004.03.001,
    author  = {Mattsson, Ken and Nordstr\"om, Jan},
    title   = {Summation by parts operators for finite difference approximations of second derivatives},
    journal = {Journal of Computational Physics},
    volume  = {199},
    number  = {2},
    pages   = {503-540},
    year    = {2004},
    month   = {09},
    doi     = {10.1016/j.jcp.2004.03.001},
    issn    = {0021-9991}
}

@article{10.1016/j.jcp.2014.01.013,
    author  = {Duchemin, Laurent and Eggers, Jens},
    title   = {The {E}xplicit--{I}mplicit--{N}ull method: Removing the numerical instability of {PDE}s},
    journal = {Journal of Computational Physics},
    volume  = {263},
    pages   = {37-52},
    year    = {2014},
    month   = {04},
    doi     = {10.1016/j.jcp.2014.01.013},
    issn    = {0021-9991}
}

@article{10.1016/j.jcp.2019.108921,
    author  = {Arbogast, Todd and Huang, Chieh-Sen and Zhao, Xikai},
    title   = {Finite volume {WENO} schemes for nonlinear parabolic problems with degenerate diffusion on non-uniform meshes},
    journal = {Journal of Computational Physics},
    volume  = {399},
    pages   = {108921},
    year    = {2019},
    month   = {12},
    doi     = {10.1016/j.jcp.2019.108921},
    issn    = {0021-9991}
}

@article{10.1016/j.jcp.2022.111314,
    author  = {Tan, Meiqi and Cheng, Juan and Shu, Chi-Wang},
    title   = {Stability of high order finite difference and local discontinuous {G}alerkin schemes with explicit-implicit-null time-marching for high order dissipative and dispersive equations},
    journal = {Journal of Computational Physics},
    volume  = {464},
    pages   = {111314},
    year    = {2022},
    month   = {09},
    doi     = {10.1016/j.jcp.2022.111314},
    issn    = {0021-9991}
}

@book{10.1016/S0168-2024_08_70184-X,
    author    = {Ciarlet, Philippe G.},
    title     = {The Finite Element Method for Elliptic Problems},
    edition   = {1},
    series    = {Studies in Mathematics and Its Applications},
    volume    = {4},
    year      = {1978},
    publisher = {North Holland},
    address   = {Amsterdam},
    isbn      = {978-0-444-85028-7},
    issn      = {0168-2024},
    numpages  = {529}
}

@article{10.1016/S0168-92749700056-1,
    author  = {Ascher, Uri Michael and Ruuth, Steven J. and Spiteri, Raymond John},
    title   = {Implicit-explicit {R}unge--{K}utta methods for time-dependent partial differential equations},
    journal = {Applied Numerical Mathematics},
    volume  = {25},
    number  = {2-3},
    pages   = {151-167},
    year    = {1997},
    month   = {11},
    doi     = {10.1016/S0168-9274(97)00056-1},
    issn    = {0168-9274}
}

@incollection{10.1016/S1570-86590007005-8,
    author    = {Eymard, Robert and Gallou\"et, Thierry and Herbin, Rapha\`ele},
    title     = {Finite volume methods},
    booktitle = {Solution of Equation in $\mathbb R^n$ (Part 3), Techniques of Scientific Computing (Part 3)},
    series    = {Handbook of Numerical Analysis},
    editor    = {Ciarlet, Philippe G. and Lions, Jacques-Louis},
    volume    = {7},
    pages     = {713-1018},
    publisher = {Elsevier},
    year      = {2000},
    doi       = {10.1016/S1570-8659(00)07005-8},
    issn      = {1570-8659},
    isbn      = {978-0-444-50350-3}
}

@article{10.1051/m2an/2013138,
    author  = {Brezzi, Franco and Falk, Richard Steven and Marini, Luisa Donatella},
    title   = {Basic principles of mixed virtual element methods},
    journal = {ESAIM: Mathematical Modelling and Numerical Analysis},
    volume  = {48},
    number  = {4},
    pages   = {1227-1240},
    year    = {2014},
    doi     = {10.1051/m2an/2013138},
    issn    = {2804-7214}
}

@article{10.1051/m2an/2015067,
    author  = {Beir\~ao da Veiga, Louren{\c c}o and Brezzi, Franco and Marini, Luisa Donatella and Russo, Alessandro},
    title   = {Mixed virtual element methods for general second order elliptic problems on polygonal meshes},
    journal = {ESAIM: Mathematical Modelling and Numerical Analysis},
    volume  = {50},
    number  = {3},
    pages   = {727-747},
    year    = {2016},
    doi     = {10.1051/m2an/2015067},
    issn    = {2804-7214}
}

@article{10.1063/1.1745010,
    author  = {Richards, Lorenzo Adolph},
    title   = {Capillary conduction of liquids through porous mediums},
    journal = {Physics},
    volume  = {1},
    number  = {5},
    pages   = {318-333},
    year    = {1931},
    month   = {11},
    doi     = {10.1063/1.1745010},
    issn    = {0148-6349}
}

@article{10.1063/1.1745233,
    author  = {Muskat, Morris},
    title   = {The flow of compressible fluids through porous media and some problems in heat conduction},
    journal = {Physics},
    volume  = {5},
    number  = {3},
    pages   = {71-94},
    year    = {1934},
    month   = {03},
    doi     = {10.1063/1.1745233},
    issn    = {0148-6349}
}

@article{10.1090/mcom/3927,
    author  = {Chen, Yuan and Xing, Yulong},
    title   = {Optimal error estimates of ultra-weak discontinuous {G}alerkin methods with generalized numerical fluxes for multi-dimensional convection-diffusion and biharmonic equations},
    journal = {Mathematics of Computation},
    volume  = {93},
    number  = {349},
    pages   = {2135-2183},
    year    = {2024},
    month   = {09},
    doi     = {10.1090/mcom/3927},
    issn    = {1088-6842}
}

@article{10.1090/mcom/3950,
    author  = {Fu, Zhaohui and Tang, Tao and Yang, Jiang},
    title   = {Energy diminishing implicit-explicit {R}unge--{K}utta methods for gradient flows},
    journal = {Mathematics of Computation},
    volume  = {93},
    number  = {350},
    pages   = {2745-2767},
    year    = {2024},
    month   = {11},
    doi     = {10.1090/mcom/3950},
    issn    = {1088-6842}
}

@article{10.1090/mcom/4167,
    author  = {Yu, Kaichang and Cheng, Juan and Liu, Yuanyuan and Shu, Chi-Wang},
    title   = {A {L}ax-{W}endroff type theorem of doubly conservative schemes for degenerate convection-diffusion equations},
    journal = {Mathematics of Computation},
    year    = {2026},
    pages   = {to appear},
    doi     = {10.1090/mcom/4167},
    issn    = {1088-6842}
}

@article{10.1090/S0025-5718-07-02045-5,
    author  = {Cheng, Yingda and Shu, Chi-Wang},
    title   = {A discontinuous {G}alerkin finite element method for time dependent partial differential equations with higher order derivatives},
    journal = {Mathematics of Computation},
    volume  = {77},
    number  = {262},
    pages   = {699-730},
    year    = {2008},
    month   = {04},
    doi     = {10.1090/S0025-5718-07-02045-5},
    issn    = {1088-6842}
}

@article{10.1090/S0025-5718-08-02123-6,
    author  = {Cockburn, Bernardo and Dong, Bo and Guzm\'{a}n, Johnny},
    year    = {2008},
    month   = {10},
    pages   = {1887-1916},
    title   = {A superconvergent {LDG}-hybridizable {G}alerkin method for second-order elliptic problems},
    volume  = {77},
    number  = {264},
    journal = {Mathematics of Computation},
    doi     = {10.1090/S0025-5718-08-02123-6},
    issn    = {1088-6842}
}

@article{10.1103/PhysRev.37.405,
    author  = {Onsager, Lars},
    title   = {Reciprocal relations in irreversible processes. {I}.},
    journal = {Physical Review},
    volume  = {37},
    number  = {4},
    pages   = {405-426},
    year    = {1931},
    month   = {02},
    doi     = {10.1103/PhysRev.37.405},
    issn    = {1536-6065}
}

@article{10.1103/PhysRev.38.2265,
    author  = {Onsager, Lars},
    title   = {Reciprocal relations in irreversible processes. {II}.},
    journal = {Physical Review},
    volume  = {38},
    number  = {12},
    pages   = {2265-2279},
    year    = {1931},
    month   = {12},
    doi     = {10.1103/PhysRev.38.2265},
    issn    = {1536-6065}
}

@article{10.1103/PhysRev.89.977,
    author  = {Spitzer, Lyman and H\"arm, Richard},
    title   = {Transport phenomena in a completely ionized gas},
    journal = {Physical Review},
    volume  = {89},
    number  = {5},
    pages   = {977-981},
    year    = {1953},
    month   = {03},
    doi     = {10.1103/PhysRev.89.977},
    issn    = {1536-6065}
}

@article{10.1137/070706616,
    author  = {Cockburn, Bernardo and Gopalakrishnan, Jayadeep and Lazarov, Raytcho Dimitrov},
    title   = {Unified Hybridization of Discontinuous {G}alerkin, Mixed, and Continuous {G}alerkin Methods for Second Order Elliptic Problems},
    journal = {SIAM Journal on Numerical Analysis},
    volume  = {47},
    number  = {2},
    pages   = {1319-1365},
    year    = {2009},
    month   = {02},
    doi     = {10.1137/070706616},
    issn    = {1095-7170}
}

@article{10.1137/0719052,
    author  = {Arnold, Douglas Norman},
    title   = {An interior penalty finite element method with discontinuous elements},
    journal = {SIAM Journal on Numerical Analysis},
    volume  = {19},
    number  = {4},
    pages   = {742-760},
    year    = {1982},
    month   = {08},
    doi     = {10.1137/0719052},
    issn    = {1095-7170}
}

@article{10.1137/0732037,
    author  = {Ascher, Uri Michael and Ruuth, Steven J. and Wetton, Brian Thomas Robinson},
    title   = {Implicit-explicit methods for time-dependent partial differential equations},
    journal = {SIAM Journal on Numerical Analysis},
    volume  = {32},
    number  = {3},
    pages   = {797-823},
    year    = {1995},
    month   = {06},
    doi     = {10.1137/0732037},
    issn    = {1095-7170}
}

@article{10.1137/0915089,
    author  = {Shen, Jie},
    title   = {Efficient spectral-{G}alerkin method {I}. Direct solvers of second- and fourth-order equations using {L}egendre polynomials},
    journal = {SIAM Journal on Scientific Computing},
    volume  = {15},
    number  = {6},
    pages   = {1489-1505},
    year    = {1994},
    month   = {11},
    doi     = {10.1137/0915089},
    issn    = {1095-7197}
}

@book{10.1137/1.9780898717938,
    author    = {Strikwerda, John Charles},
    title     = {Finite Difference Schemes and Partial Differential Equations},
    edition   = {2},
    series    = {Other Titles in Applied Mathematics},
    volume    = {88},
    publisher = {Society for Industrial and Applied Mathematics},
    address   = {Philadelphia, PA},
    year      = {2004},
    doi       = {10.1137/1.9780898717938},
    isbn      = {978-0-89871-793-8}
}

@article{10.1137/100791002,
    author  = {Liu, Yuanyuan and Shu, Chi-Wang and Zhang, Mengping},
    title   = {High order finite difference {WENO} schemes for nonlinear degenerate parabolic equations},
    journal = {SIAM Journal on Scientific Computing},
    volume  = {33},
    number  = {2},
    pages   = {939-965},
    year    = {2011},
    month   = {01},
    doi     = {10.1137/100791002},
    issn    = {1095-7197}
}

@article{10.1137/110839230,
    author  = {Zhang, Xiangxiong and Liu, Yuanyuan and Shu, Chi-Wang},
    title   = {Maximum-principle-satisfying high order finite volume weighted essentially nonoscillatory schemes for convection-diffusion equations},
    journal = {SIAM Journal on Scientific Computing},
    volume  = {34},
    number  = {2},
    pages   = {A627-A658},
    year    = {2012},
    month   = {01},
    doi     = {10.1137/110839230},
    issn    = {1095-7197}
}

@article{10.1137/S0036142901384162,
    author  = {Arnold, Douglas Norman and Brezzi, Franco and Cockburn, Bernardo and Marini, Luisa Donatella},
    title   = {Unified Analysis of Discontinuous {G}alerkin Methods for Elliptic Problems},
    journal = {SIAM Journal on Numerical Analysis},
    volume  = {39},
    number  = {5},
    pages   = {1749-1779},
    year    = {2002},
    month   = {01},
    doi     = {10.1137/S0036142901384162},
    issn    = {1095-7170}
}

@article{10.1137/S0036142997316712,
    author  = {Cockburn, Bernardo and Shu, Chi-Wang},
    title   = {The Local Discontinuous {G}alerkin Method for Time-Dependent Convection-Diffusion Systems},
    journal = {SIAM Journal on Numerical Analysis},
    volume  = {35},
    number  = {6},
    pages   = {2440-2463},
    year    = {1998},
    month   = {11},
    doi     = {10.1137/S0036142997316712},
    issn    = {1095-7170}
}

@article{10.4208/csiam-am.SO-2025-0095,
    author  = {Liao, Hong-lin and Wang, Xuping and Wen, Cao},
    title   = {Average energy dissipation rates of implicit-explicit {R}unge-{K}utta methods for gradient flow problems},
    journal = {CSIAM Transactions on Applied Mathematics},
    volume  = {8},
    number  = {1},
    pages   = {36-76},
    year    = {2027},
    month   = {06},
    doi     = {10.4208/csiam-am.SO-2025-0095},
    issn    = {2708-0579}
}

@article{10.4208/eajam.2022-271.090123,
    author  = {Tan, Meiqi and Cheng, Juan and Shu, Chi-Wang},
    title   = {Stability of spectral collocation schemes with explicit-implicit-null time-marching for convection-diffusion and convection-dispersion equations},
    journal = {East Asian Journal on Applied Mathematics},
    volume  = {13},
    number  = {3},
    pages   = {464-498},
    year    = {2023},
    month   = {06},
    doi     = {10.4208/eajam.2022-271.090123},
    issn    = {2079-7370}
}

@article{10.4208/jcm.1307-m4273,
    author  = {Vidden, Chad and Yan, Jue},
    title   = {A new direct discontinuous {G}alerkin method with symmetric structure for nonlinear diffusion equations},
    journal = {Journal of Computational Mathematics},
    volume  = {31},
    number  = {6},
    pages   = {638-662},
    year    = {2013},
    month   = {10},
    doi     = {10.4208/jcm.1307-m4273},
    issn    = {1991-7139}
}

@article{10.4208/nmtma.OA-2024-0076,
    author  = {Li, Yumiao and Yang, Yin and Liu, Tiegang and Yuan, Weixiong and Cao, Kui},
    title   = {The direct discontinuous {G}alerkin method with explicit-implicit-null time discretizations for nonlinear diffusion equations},
    journal = {Numerical Mathematics: Theory, Methods and Applications},
    volume  = {18},
    number  = {1},
    pages   = {175-198},
    year    = {2025},
    month   = {04},
    doi     = {10.4208/nmtma.OA-2024-0076},
    issn    = {2079-7338}
}

@inproceedings{BassiEtAl1997BR2,
    author    = {Bassi, Francesco and Rebay, Stefano and Mariotti, Gianluca and Pedinotti, Savini and Savini, Marco},
    title     = {A high-order accurate discontinuous finite element method for inviscid and viscous turbomachinery flows},
    booktitle = {Proceedings of 2nd European Conference on Turbomachinery, Fluid Dynamics and Thermodynamics},
    editor    = {Decuypere, Roland and Dibelius, G\"unther},
    publisher = {Technologisch Instituut},
    address   = {Antwerpen, Belgium},
    pages     = {99-108},
    year      = {1997},
    month     = {03},
    issn      = {2313-0067}
}

@inproceedings{BrezziEtAl1999Diffusion,
    author    = {Brezzi, Franco and Manzini, Marco and Marini, Luisa Donatella and Pietra, Paola and Russo, Alessandro},
    title     = {Discontinuous finite elements for diffusion problems},
    booktitle = {Francesco Brioschi (1824--1897) Convegno di Studi Matematici},
    series    = {Incontro di Studio},
    number    = {16},
    publisher = {Istituto Lombardo, Accademia di Scienze e Lettere},
    address   = {Milano},
    pages     = {197-217},
    year      = {1999},
}

@book{DGM1962,
    author    = {de Groot, Sybren Ruurds and Mazur, Peter},
    title     = {Non-Equilibrium Thermodynamics},
    publisher = {North-Holland},
    address   = {Amsterdam},
    year      = {1962}
}

@book{Kirchhoff1894,
    author    = {Kirchhoff, Gustav Robert},
    title     = {Vorlesungen \"uber die Theorie der W\"arme},
    series    = {Vorlesungen \"uber mathematische Physik},
    volume    = {4},
    year      = {1894},
    publisher = {Teubner},
    address   = {Leipzig}
}

@unpublished{YCS2026DoCFV,
    author = {Yu, Kaichang and Cheng, Juan and Shu, Chi-Wang},
    title  = {A high-order doubly conservative weighted essentially non-oscillatory finite volume scheme for degenerate convection-diffusion equations on unstructured meshes},
    year   = {2026},
    note   = {Submitted to SIAM Journal on Scientific Computing},
    issn   = {1095-7197}
}

\end{document}